\documentclass[preprint,12pt]{amsart}

\usepackage[margin=1.5in]{geometry}

\usepackage{amsmath}
\usepackage{amsfonts}
\usepackage{amssymb}
\usepackage{xypic}
\usepackage{longtable}
\usepackage{amsthm}

\theoremstyle{plain}
\newtheorem{theorem}{Theorem}[section]
\newtheorem{corollary}[theorem]{Corollary}
\newtheorem{lemma}[theorem]{Lemma}
\newtheorem{proposition}[theorem]{Proposition}
\theoremstyle{definition}

\makeatletter
\def\ps@pprintTitle{%
  \let\@oddhead\@empty
  \let\@evenhead\@empty
  \let\@oddfoot\@empty
  \let\@evenfoot\@oddfoot
}
\makeatother

\title{Galois subspaces for the rational normal curve}
\author{Robert Auffarth and Sebasti\'an Rahausen}
\address{R. Auffarth \\Departamento de Matem\'aticas, Facultad de
Ciencias, Universidad de Chile, Santiago\\Chile}
\email{rfauffar@uchile.cl}
\address{S. Rahausen \\Departamento de Matem\'aticas, Facultad de
Ciencias, Universidad de Chile, Santiago\\Chile}
\email{srahausen@gmail.com}

\thanks{Both authors were partially supported by CONICYT PIA ACT1415.}


\begin{document}

\maketitle

\begin{abstract}
We characterize all $(n-2)$-dimensional linear subspaces of $\mathbb{P}^{n}$ such that the induced linear projection, when restricted to the rational normal curve, gives a Galois morphism. We give an explicit description of these spaces as a disjoint union of locally closed subvarieties in the Grassmannian $\mathbb{G}(n-2,n)$.
\end{abstract}

\section{Introduction}

Given an embedding $X\hookrightarrow\mathbb{P}^n$ of a $d$-dimensional smooth projective variety into projective space and a linear subspace $W\in\mathbb{G}(n-d-1,n)$, it is natural to ask about the monodromy of the linear projection $\pi:X\dashrightarrow\mathbb{P}^d$ from $X$ with center $W$, and in particular if it induces a Galois extension of function fields. The monodromy of linear projections from varieties embedded in projective space has been the focus of several articles (see for example \cite{Cuk} and \cite{Pirola}), and in the last decade the question of when a linear projection is Galois has been intensely studied under the name of \emph{Galois embeddings}. 

The formal definition of Galois embedding was introduced by Yoshihara in \cite{Yoshi}. We recall that if $X$ is a smooth projective variety of dimension $d$ and $D$ is a very ample divisor on $X$ that induces an embedding $\varphi_D:X\to\mathbb{P}^n$, then the embedding is said to be \emph{Galois} if there exists a linear subspace  $W\in\mathbb{G}(n-d-1,n)$ such that $W\cap\varphi_D(X)=\varnothing$ and the linear projection $\pi_W:\mathbb{P}^n\dashrightarrow\mathbb{P}^d$ defined by $W$ restricts to a morphism $\pi:X\to\mathbb{P}^d$ that induces a Galois extension of function fields $\pi^*(k(\mathbb{P}^d))\subseteq k(X)$. The calculation of Galois subspaces for different varieties has been an object of interest in the past few years, and especially in the case of curves and abelian varieties. See, for instance, \cite{Auff}, \cite{Cuk}, \cite{Pirola}, \cite{Takahashi}, \cite{Yoshi} and the references therein. The purpose of this article is to classify all linear subspaces of $\mathbb{P}^n$ that induce a Galois morphism for an embedding of $\mathbb{P}^1$ into $\mathbb{P}^n$, thereby answering a question asked by Yoshihara in his list of open problems \cite{Yoshiproblems}. This question was already addressed by Yoshihara himself for $n=3$ in \cite[Prop. 4.1]{Yoshi2}, and our article can be seen as a generalization of his results to arbitrary dimension.

In the case of $\mathbb{P}^1$, the only embedding of $\mathbb{P}^1$ into $\mathbb{P}^n$ such that the image is not contained in a hyperplane, modulo a linear change of coordinates, is the Veronese embedding $\nu_n:\mathbb{P}^1\to\mathbb{P}^n$ which is always Galois as we will see. Given a linear projection $\pi_W:\mathbb{P}^n\dashrightarrow\mathbb{P}^1$ with center $W$, one can ask about the Galois closure of the field extension induced by $\pi=\pi_W\circ \nu_n:\mathbb{P}^1\to\mathbb{P}^1$. The Galois group of this closure is the monodromy group of $\pi$, and is related to several hard and interesting problems; see for example \cite{Adrianov}, \cite{Cuk}, \cite{GS}, \cite{Konig}, \cite{Muller} and \cite{Muller2} among many others. 

By \cite[Theorem 2.2]{Yoshi}, an embedding of $X$ into $\mathbb{P}^n$ is Galois if and only if there exists a finite group of automorphisms $G$ of $X$ such that $|G|=(D^n)$, there exists a $G$-invariant linear subspace $\mathcal{L}$ of $H^0(X,\mathcal{O}_X(D))$ of dimension $d+1$ such that $G$ acts on $\mathbb{P}\mathcal{L}$ as the identity, and such that the linear system $\mathcal{L}$ has no base points. In practice, given a group $G$ and a $G$-invariant divisor $D$, one must find the subspace $\mathcal{L}$; this is the approach we will take in this article.

Before stating our main theorem, we define the linear subspaces we will be working with. For $a,b,c,d\in k$ we define the homogeneous linear polynomial in the variables $x_0,\ldots,x_n$
$$L_{a,c}:=\sum_{i=0}^n\binom{n}{i}(-1)^ia^ic^{n-i}x_i.$$
If $\alpha\in k$ and $n=2m$ is even, we define the polynomials
$$A_{a,b,c,d}:=\sum_{i+j+l=m}\binom{m}{i,j,l}(-1)^i(ad+bc)^i(ab)^j(cd)^lx_{i+2l}$$
$$B_{a,b,c,d}^\alpha:=\sum_{i=0}^{n}\binom{n}{i}(-1)^i(\alpha^ma^{n-i}c^i+b^{n-i}d^i)x_i.$$
We note that the hyperplane defined by $L_{a,c}$ only depends on the class of $(a,c)$ in $\mathbb{P}^1$ when $(a,c)$ is non-zero, and we will therefore denote it by $H_{[a:c]}$. The linear subspace defined by $A_{a,b,c,d}=B_{a,b,c,d}^\alpha=0$ for $(a,b,c,d)\neq(0,0,0,0)$ only depends on the class of $(a,b,c,d)$ in $\mathbb{P}^3$ and on $\alpha$ and will be denoted by $V_{[a:b:c:d]}^\alpha$.

\begin{theorem}\label{main}
Let $k$ be a field, let $\nu_n:\mathbb{P}^1\to\mathbb{P}^n$ be the Veronese embedding of degree $n\geq3$ over $k$, and assume that the characteristic of $k$ does not divide $n$ and that $k$ contains a primitive $n$th root of unity $\zeta_n$. Then we have the following:
\begin{enumerate}
\item For each $([a:c],[b:d])\in(\mathbb{P}^1\times\mathbb{P}^1)\backslash\Delta$,
$$W_{[a:c],[b:d]}:=H_{[a:c]}\cap H_{[b:d]}$$
 is a Galois subspace of $\mathbb{P}^n$ for $\nu_n$ with cyclic Galois group that is conjugate to the group generated by $[x:y]\mapsto[\zeta_nx:y]$. This gives a 2-dimensional locally closed subvariety $\mathcal{C}_n$ of $\mathbb{G}(n-2,n)$.
 
\item If $n=2m$ is even, in addition to the previous item, for each representative $\alpha\in k^\times/(k^\times)^2\mu_m(k)$ we have the Galois subspaces $V_{[a:b:c:d]}^\alpha$ for $[a:b:c:d]\in\mathbb{P}^3$ and $ad-bc\neq0$. Here the Galois group is the dihedral group of order $2m$ and is conjugate to the group generated by $[x:y]\mapsto[\zeta_nx:y]$ and $[x:y]\mapsto[\alpha y:x]$. Thus we obtain families $\mathcal{D}_m^\alpha$ which over $\overline{k}$ give a 3-dimensional locally closed subvariety $\mathcal{D}_m$ of $\mathbb{G}(n-2,n)$. 
\end{enumerate}

The above Galois subspaces are all disjoint from $\nu_n(\mathbb{P}^1)$. Moreover, if $n\notin\{4,12,24,60\}$, these are the only Galois subspaces of $\mathbb{P}^n$ for $\nu_n$ disjoint from $\nu_n(\mathbb{P}^1)$.
\end{theorem}

The exceptional cases $n\in\{2,4,12,24,60\}$ are addressed in the following proposition:

\begin{proposition}\label{exceptional}
We have the following families of Galois subspaces disjoint from $\nu_n(\mathbb{P}^1)$:

\begin{enumerate}
 \item For $n=2$, every element of $\mathbb{G}(0,2)\backslash \nu_2(\mathbb{P}^1)$ gives a Galois subspace for $\nu_2(\mathbb{P}^1)$.
 \item For $n=4$, apart from $\mathcal{C}_4$, for every $\beta\in k^\times/(k^\times)^2$ there is a family $\mathcal{K}^\beta\subseteq\mathbb{G}(2,4)$ that corresponds to Galois subspaces whose Galois group is the Klein four group. Over $\overline{k}$ these give a $3$-dimensional locally closed subvariety $\mathcal{K}$ of $\mathbb{G}(2,4)$.
 \end{enumerate}
 
 If, moreover, $-1$ is the sum of two squares in $k$, then:
 \begin{enumerate}
 \item[(3)] For $n=12$, apart from $\mathcal{C}_{12}$ and the spaces $\mathcal{D}_{12}^\alpha$, there exists a 3-dimensional family $\mathcal{A}_4\subseteq\mathbb{G}(10,12)$ corresponding to Galois subspaces with Galois group $A_4$.
 \item[(4)] For $n=24$, apart from $\mathcal{C}_{24}$ and the spaces $\mathcal{D}_{24}^\alpha$, there exists a 3-dimensional family $\mathcal{S}_4\subseteq\mathbb{G}(22,24)$ corresponding to Galois subspaces with Galois group $S_4$.
 \item[(5)] For $n=60$, if $5$ is a square in $k$, then apart from $\mathcal{C}_{60}$ and the spaces $\mathcal{D}_{60}^\alpha$, there exists a 3-dimensional family $\mathcal{A}_5\subseteq\mathbb{G}(58,60)$ corresponding to Galois subspaces with Galois group $A_5$.
\end{enumerate}
 
\end{proposition}

The previous two results exhaust the list of possible Galois subspaces in the case they are disjoint from $\nu_n(\mathbb{P}^1)$. The case when the Galois subspace is not disjoint from $\nu_n(\mathbb{P}^1)$ will be analyzed in Section \ref{Nondisjoint} where we give a complete explicit classification of all Galois subspaces for the Veronese embedding in any degree in the case that the base field is algebraically closed. In particular, we show that for $n\geq60$ there are $n+\lfloor\frac{n}{2}\rfloor+2$ disjoint families of Galois subspaces in $\mathbb{G}(n-2,n)$, and the biggest family is of dimension $n$. See Table \ref{classification} for the full classification.

\section{Preliminaries}

In this section we recall some classical results, along with results we will need for Galois embeddings. Let $\text{End}_n(\mathbb{P}^1)$ denote the set of endomorphisms of $\mathbb{P}^1$ of degree $n$. We see that $\mbox{Aut}(\mathbb{P}^1)$ acts on $\text{End}_n(\mathbb{P}^1)$ via composition. Let 
$$\nu_n:\mathbb{P}^1\to\mathbb{P}^n$$
denote the $n$th Veronese embedding.

\begin{lemma}\label{Bij}
There is a bijective correspondence between $\text{End}_n(\mathbb{P}^1)/\text{Aut}(\mathbb{P}^1)$ and all linear subspaces $W\in\mathbb{G}(n-2,n)$ such that $W\cap \nu_n(\mathbb{P}^1)=\varnothing$.
\end{lemma}
\begin{proof}
Let $x,y$ be the coordinates on $\mathbb{P}^1$. Then a degree $n$ endomorphism $f$ is given by $[p(x,y):q(x,y)]$ where $p$ and $q$ are homogeneous polynomials of degree $n$ without common factors. Write
$$p(x,y)=a_0x^n+a_1x^{n-1}y+\cdots+a_ny^n$$
$$q(x,y)=b_0x^n+b_1x^{n-1}y+\cdots+b_ny^n.$$
We then send $f$ to the linear subspace 
$$W_f:=\{a_0x_0+\cdots+a_nx_n=b_0x_0+\cdots+b_nx_n=0\}\in\mathbb{G}(n-2,n).$$
If $W_f\cap \nu_n(\mathbb{P}^1)\neq\varnothing$, then there exists $[x:y]\in\mathbb{P}^1$ such that
$$a_0x^n+\cdots+a_ny^n=b_0x^n+\cdots+b_ny^n=0.$$
This, however, implies that the polynomials $p(x,y)$ and $q(x,y)$ share a common factor, and thus $f$ is not of degree $n$. Therefore $W_f\cap \nu_n(\mathbb{P}^1)=\varnothing$.

Now take $W\in\mathbb{G}(n-2,n)$ such that $W\cap \nu_n(\mathbb{P}^1)=\varnothing$, let $L_0,L_1\in k[x_0,\ldots,x_n]_1$ be two defining equations for $W$, and consider the morphism $\pi=[L_0:L_1]\circ \nu_n:\mathbb{P}^1\to\mathbb{P}^1$. We first note that the sections $\nu_n^*L_0,\nu_n^*L_1$ do not have common roots, since that would contradict the fact that $W\cap \nu_n(\mathbb{P}^1)=\varnothing$. Secondly, these forms have degree $n$ since $\nu_n(\mathbb{P}^1)$ is of degree $n$ in $\mathbb{P}^n$. We see that different defining equations for $W$ differ by an action of $\text{PGL}(2,k)$, and this translates to composing the morphism $\pi$ by an automorphism of $\mathbb{P}^1$.
\end{proof}

The following corollary is clear from the previous proof:

\begin{corollary}
The previous correspondence gives a bijection between Galois morphisms of degree $n$ from $\mathbb{P}^1$ to $\mathbb{P}^1$ modulo the action of $\text{Aut}(\mathbb{P}^1)$, and Galois subspaces of $\mathbb{P}^n$ for $\nu_n$, disjoint from $\nu_n(\mathbb{P}^1)$.
\end{corollary}

This shows that characterizing Galois subspaces of $\mathbb{P}^n$ for $\nu_n$ is essentially the same as characterizing Galois morphisms from $\mathbb{P}^1$ to itself, and therefore the same as characterizing rational subfields of $k(t)$ that give Galois extensions.

Let $G$ be a finite group acting (effectively) on $\mathbb{P}^1$ such that its order is not divisible by the characteristic of the ground field $k$. By a close study of the Riemann-Hurwitz formula and the ramification points that may appear for this action, one can conclude that $G$ is either $\mathbb{Z}/n\mathbb{Z}$ for any $n\in\mathbb{N}$, the dihedral group $D_m$ of order $2m$ for any $m\in\mathbb{N}$, $A_4$, $S_4$ or $A_5$. Over the complex numbers, these last three groups appear as the (holomorphic) automorphism groups of the Platonic solids which can be inscribed in the Riemann sphere. Not all of these groups always appear, however, for an arbitrary field. Indeed, we have the following (cf. \cite[Prop. 1.1]{Beau}) which shows the conditions that need to be met in Proposition \ref{exceptional}:

\begin{proposition}
\begin{enumerate}
\item $\text{PGL}(2,k)$ contains $\mathbb{Z}/n\mathbb{Z}$ and the dihedral group $D_n$ if and only if $k$ contains $\zeta+\zeta^{-1}$ for some primitive $n$th root of unity $\zeta$.
\item $\text{PGL}(2,k)$ contains $A_4$ and $S_4$ if and only if $-1$ is the sum of two squares in $k$.
\item $\text{PGL}(2,k)$ contains $A_5$ if and only if $-1$ is the sum of two squares and $5$ is a square in $k$.
\end{enumerate}
\end{proposition}

By \cite[Theorem 4.2]{Beau}, two isomorphic finite groups of automorphisms are conjugate by automorphisms of the projective line, except for $G\simeq\mathbb{Z}/2\mathbb{Z}$, $G$ the Klein four group, or $G$ the dihedral group. We will look at this further on.

Let $D\in\mbox{Div}(\mathbb{P}^1)$ be a divisor of positive degree (which for $\mathbb{P}^1$ is equivalent to being very ample), and let $\varphi_D:\mathbb{P}^1\to\mathbb{P}^n$ be the associated embedding, where $n=\deg(D)$. Since $D$ is linearly equivalent to $n[1:0]$, we have that $\varphi_D$ is the Veronese embedding
$$[x:y]\mapsto[x^n:x^{n-1}y:\cdots:xy^{n-1}:y^n]$$ 
of $\mathbb{P}^1$ in $\mathbb{P}^n$ and its image is a rational normal curve. We are interested in characterizing all Galois subspaces of this embedding, thereby answering a question raised by Yoshihara (cf. \cite[(4)(b)(ii)]{Yoshiproblems}).\\
 
Let $D=\sum_{i}^rm_i[p_i:q_i]$, and consider the rational map
$$f_D:=\prod_{i}^r(q_ix-p_iy)^{-m_i}.$$
A trivial calculation shows that
\begin{eqnarray}\nonumber L(D)&:=&\{h\in k(\mathbb{P}^1):\text{div}(h)+D\geq0\}\\
\nonumber&=&\langle f_Dx^n,f_Dx^{n-1}y,\ldots,f_Dxy^{n-1},f_Dy^n\rangle\\
 \nonumber &=&f_DH^0(\mathbb{P}^1,\mathcal{O}_{\mathbb{P}^1}(n))
\end{eqnarray}

Therefore given a finite group of automorphisms $G$ of $\mathbb{P}^1$ that preserve the divisor $D$, in order to find the Galois subspace of $\mathbb{P}^n$ associated to the group $G$ we must understand the action of $G$ on $L(D)$ by pullback. Indeed, we are looking for a 2-dimensional subspace $\mathcal{L}\subseteq L(D)$ such that $G$ acts as the identity on $\mathbb{P}\mathcal{L}$.

In what follows we will divide our analysis according to the group that is acting.

\section{Cyclic group action}

Let $n\geq3$, assume that $k$ contains a primitive $n$th root of unity $\zeta_n$, and consider the order $n$ morphism 
$$C_n:[x:y]\mapsto[\zeta_nx:y].$$ 
By \cite[Theorem 4.2]{Beau}, we have that every cyclic subgroup of $\text{PGL}(2,k)$ is conjugate to $\langle C_n\rangle$. Take 
$$M:=\left(\begin{array}{cc}a&b\\c&d\end{array}\right)\in\text{PGL}(2,k).$$
We have that
$$T_n:=MC_nM^{-1}=\left(\begin{array}{cc}ad\zeta_n-bc&ab(1-\zeta_n)\\cd(\zeta_n-1)&ad-bc\zeta_n\end{array}\right)$$
and $T_n$ leaves the divisor $D:=n[a:c]=nM([1:0])$ invariant. By the previous section, we have that
$$L(D)=\langle f_Dx^n,f_Dx^{n-1}y,\ldots,f_Dxy^{n-1},f_Dy^n\rangle$$
where 
$$f_D=(cx-ay)^{-n}.$$
We see that
$$1=(cx-ay)^{n}f_D\in L(D)$$
and is trivially invariant under the action of $T_n$. Notice that the automorphism 
$$R:=\left(\begin{array}{cc}d&-b\\-c&a\end{array}\right)$$
sends $[b:d]$ to $[0:1]$ and $[a:c]$ to $[1:0]$; moreover $RT_n$ does the same. This implies that $RT_n$ is a multiple of $R$, and upon further inspection it is easy to see that this multiple is an $n$th root of unity. Therefore
$$(dx-by)^{n}f_D\in L(D)$$
is also invariant under composition by $T_n$. This means that the subspace of $L(D)$ generated by 
$$1=f_D\sum_{i=0}^n\binom{n}{i}(-1)^ia^ic^{n-i}x^{n-i}y^i$$
$$(dx-by)^{n}f_D=f_D\sum_{i=0}^n\binom{n}{i}(-1)^ib^id^{n-i}x^{n-i}y^i$$
is the eigenspace of $T_n$ for the eigenvalue 1. Therefore, using the notation in the introduction, we obtain the Galois subspaces
$$H_{[a:c],[b:d]}:=\{L_{a,c}=L_{b,d}=0\}.$$
These are the only Galois subspaces (disjoint from $\nu_n(\mathbb{P}^1)$) with cyclic Galois group and give a family $\mathcal{C}_n\subseteq\mathbb{G}(n-2,n)$; we will verify that the dimension of this family is 2 in Section \ref{ex}.

\section{Dihedral group action}

Let $n=2m\geq6$ be even, assume that $k$ contains a primitive $n$th root of unity $\zeta_n$, and set $\zeta_m:=\zeta_n^2$. For $\alpha\in k^\times/(k^\times)^2\mu_m(k)$ consider the morphisms 
$$C_m:[x:y]\mapsto[\zeta_mx:y]$$
$$I_\alpha:[x:y]\mapsto[\alpha y:x].$$
By \cite[Theorem 4.2]{Beau}, we have that every subgroup of $\text{PGL}(2,k)$ isomorphic to the dihedral group of order $2m$ is conjugate to $\langle C_m,I_\alpha\rangle$ for some $\alpha$, and the conjugacy classes of the dihedral group in $\text{PGL}(2,k)$ are parametrized by $\alpha\in k^\times/(k^\times)\mu_m(k)$.

Taking the same matrix
$$M=\left(\begin{array}{cc}a&b\\c&d\end{array}\right)\in\text{PGL}(2,k)$$
as before, we see that
$$T_m=MC_mM^{-1}=\left(\begin{array}{cc}ad\zeta_m-bc&ab(1-\zeta_m)\\cd(\zeta_m-1)&ad-bc\zeta_m\end{array}\right)$$
$$J_\alpha:=MI_\alpha M^{-1}=\left(\begin{array}{cc}bd-\alpha ac& \alpha a^2-b^2\\d^2-\alpha c^2&\alpha ac-bd\end{array}\right).$$
Now we take the divisor $D=m[a:c]+m[b:d]$ which is invariant under these two transformations. We set
$$f_D=(cx-ay)^{-m}(dx-by)^{-m},$$
and use the basis $\{f_Dx^n,f_Dx^{n-1}y,\ldots,f_Dxy^{n-1},f_Dy^n\}$ of $L(D)$. Using the same idea as the previous section, we see that
$$1=(cx-ay)^{m}(dx-by)^{m}f_D\in L(D)$$
is fixed by the action of $T_m$ and $J_\alpha$. With a little work, it is not difficult to show that 
$$\left(\alpha^m(cx-ay)^n+(dx-by)^n\right)f_D$$
is also fixed by both transformations. Therefore after expanding, we obtain that the group generated by $T_m$ and $J_\alpha$ acts with the trivial representation on the space generated by
$$f_D\sum_{i+j+l=m}\binom{m}{i,j,l}(-1)^i(ad+bc)^i(ab)^j(cd)^lx^{i+2l}y^{i+2j}$$
$$f_D\sum_{i=0}^{n}\binom{n}{i}(-1)^i(\alpha^ma^{n-i}c^i+b^{n-i}d^i)x^{i}y^{n-i}.$$
This implies that the linear space 
$$V_{[a:b:c:d]}^\alpha:=\{A_{a,b,c,d}=B_{a,b,c,d}^\alpha=0\}$$
is a Galois subspace disjoint from $\nu_n(\mathbb{P}^1)$ with dihedral Galois group. This gives a family $\mathcal{D}_m^\alpha\subseteq\mathbb{G}(n-2,n)$ whose dimension we will verify to be 3 in the next section.

\section{Exceptional cases} \label{ex}

We will now look at the remaining cases (i.e. $n\in\{2,4,12,24,60\}$); assume here that $k$ is algebraically closed. For $n=2$, each linear projection composed with the Veronese embedding is either of degree 2 or 1, depending on whether the center of projection is disjoint from $\nu_2(\mathbb{P}^1)$ or not. In either case, the induced field extension is of the same degree, and is therefore Galois. Therefore every element of $\mathbb{G}(0,2)=\mathbb{P}^2$ induces a Galois morphism.

For $n\in\{4,12,24,60\}$, we invite the interested reader to explicitly calculate the Galois subspaces as done above in the cyclic and dihedral cases. Although tedious, the calculations are elementary. We will proceed to show that in each of these cases, for the Klein four group $K_4$, $A_4$, $S_4$ and $A_5$ there is a locally closed 3-dimensional family of Galois subspaces in the respective Grassmannian.

We note that $\mbox{PGL}(2,k)$ is 3-dimensional, and the conjugacy class of a given subgroup $G\leq\text{PGL}(2,k)$ is parametrized by $\mbox{PGL}(2,k)/N_G$, where $N_G$ is the normalizer of $G$. If $C_G$ is the centralizer of $G$ in $\text{PGL}(2,k)$, we have a natural homomorphism $\varphi:N_G\to\text{Aut}(G)$ whose kernel is $C_G$. In particular, if $G$ is finite, then $C_G$ is of finite index in $N_G$. By the table on page 27 of \cite{Beau}, we have that the centralizers of $K_4$, $D_n$, $A_4$, $S_4$ and $A_5$ in $\text{PGL}(2,k)$ are all finite, and therefore the normalizers are also finite. In the case of $\mathbb{Z}/n\mathbb{Z}$, the centralizer is 1-dimensional and therefore the normalizer is as well. By \cite[Theorem 4.2]{Beau}, since we are assuming $k$ is algebraically closed, there is only one conjugacy class of each of finite subgroup of $\text{PGL}(2,k)$.

\begin{proposition}
The families $\mathcal{C}_n$ are 2-dimensional and the families $\mathcal{D}_{m}$ are 3-dimensional for all $n,m\geq 3$. For $K_4$, $A_4$, $S_4$ and $A_5$ we obtain families of Galois subspaces $\mathcal{K}\subseteq\mathbb{G}(2,4)$, $\mathcal{A}_4\subseteq\mathbb{G}(10,12)$, $\mathcal{S}_4\subseteq\mathbb{G}(22,24)$ and $\mathcal{A}_5\subseteq\mathbb{G}(58,60)$, respectively, that are all 3-dimensional.
\end{proposition}

This completes the proof of Theorem \ref{main} and Proposition \ref{exceptional}.

\section{Non-disjoint Galois subspaces}\label{Nondisjoint}

For this section we will assume that $k$ is algebraically closed. Let $W\in\mathbb{G}(n-2,n)$ be a subspace that is not disjoint from $\nu_n(\mathbb{P}^1)$. Projection from $W$, restricted to $\nu_n(\mathbb{P}^1)$, can give a Galois morphism as in the case of $W=\{x_0=x_2=0\}$. 

However, we see by Lemma \ref{Bij} and its proof that if $W\in\mathbb{G}(n-2,n)$ and $W\cap \nu_n(\mathbb{P}^1)\neq\varnothing$, then $\pi_W \nu_n:\mathbb{P}^1\to\mathbb{P}^1$ is of degree less than $n$, where $\pi_W:\mathbb{P}^n\dashrightarrow\mathbb{P}^1$ is the projection with center $W$. As a matter of fact, by the same lemma, there exists a unique $\widetilde{W}\in\mathbb{G}(m-2,m)$ such that $\pi_{\widetilde{W}} \nu_m=\pi_W \nu_n$ (modulo automorphisms of $\mathbb{P}^1$), where $m=\deg(\pi_W \nu_n)$. We therefore obtain a morphism
$$\Phi_{n,m}:\frak{X}_{n,m}:=\{W\in\mathbb{G}(n-2,n):\deg(\pi_W\nu_n)=m\}\to\mathbb{G}(m-2,m).$$

\begin{proposition}
The fibers of $\Phi_{n,m}$ are irreducible and of dimension $n-m$. In particular $\Phi_{n,m}$ is injective if and only if $n=m$.
\end{proposition}
\begin{proof}
Let us first calculate the dimension of the fibers. We see that the image of $\Phi_{n,m}$ consists of those subspaces that are disjoint from $\nu_m(\mathbb{P}^1)$. Indeed, we have a map $i_{m,n}:\mathbb{G}(m-2,m)\hookrightarrow\mathbb{G}(n-2,n)$ that takes a subspace defined by two equations and sends it to the subspace of dimension $n-2$ defined by the same two equations; it is easy to see that $\Phi_{n,m}i_{m,n}$ is the identity on those subspaces of $\mathbb{G}(m-2,m)$ that are disjoint from $\nu_m(\mathbb{P}^1)$.

In particular, if $F$ is a general fiber of $\Phi_{n,m}$, then
$$\dim F=\dim\frak{X}_{n,m}-2(m-1).$$
Take $W\in\frak{X}_{n,m}$ and let it be the intersection of two hyperplanes $H_1,H_2\in\mathbb{G}(n-1,n)$. Write
$$\nu_n^*H_1=p_1+\cdots+p_n$$
$$\nu_n^*H_2=q_1+\cdots+q_n$$
where the above points are not necessarily different. Then since $W\in\frak{X}_{n,m}$, we have that exactly $n-m$ of the $p_i$ (counted with multiplicity) are equal to $n-m$ of the $q_i$ (counted with multiplicity). We have a rational map
$$\text{Sym}^{n-m}(\mathbb{P}^1)\times\text{Sym}^{m}(\mathbb{P}^1)\times\text{Sym}^{m}(\mathbb{P}^1)\dashrightarrow\mathbb{G}(n-2,n)$$
where
$$(D_1,D_2,D_3)\mapsto\text{span}\left\{\nu_n(D_1+D_2)\right\}\cap\text{span}\left\{\nu_n(D_1+D_3)\right\}.$$
Let $\frak{Y}_{n,m}$ be the set of all $(D_1,D_2,D_3)\in\text{Sym}^{n-m}(\mathbb{P}^1)\times\text{Sym}^{m}(\mathbb{P}^1)\times\text{Sym}^{m}(\mathbb{P}^1)$ such that $D_1+D_2\neq D_1+D_3$ and $\text{supp}(D_2)\cap\text{supp}(D_3)=\varnothing$. Then we have a morphism
$$\theta:\frak{Y}_{n,m}\to\mathbb{G}(n-2,n)$$
whose image is $\frak{X}_{n,m}$. Since the natural intersection morphism 
$$(\mathbb{G}(n-1,n)\times\mathbb{G}(n-1,n))\backslash\Delta\to\mathbb{G}(n-2,n)$$
has 2-dimensional fibers, we obtain that the fibers of the map $\theta$ are 2-dimensional. This implies that 
$$\dim\frak{X}_{n,m}=n+m-2.$$
We now conclude that a general fiber of $\Phi_{n,m}$ has dimension 
$$n+m-2-2(m-1)=n-m.$$
A more careful analysis shows that if $V\in\mathbb{G}(m-2,m)$ and is disjoint from $\nu_m(\mathbb{P}^1)$, and $(D_1,D_2,D_3)\in\theta^{-1}\Phi_{n,m}^{-1}(V)$ is any preimage, then 
$$\Phi_{n,m}^{-1}(V)=\theta(\text{Sym}^{n-m}(\mathbb{P}^1)\times\{D_2\}\times\{D_3\})$$
and is therefore irreducible.

\end{proof}

With respect to determining when an element $W\in\mathbb{G}(n-2,n)$ that is not disjoint from $\nu_n(\mathbb{P}^1)$ gives a Galois subspace, it is clear that we have the following:

\begin{proposition}
If $W\in\mathbb{G}(n-2,n)$ is such that $\deg(\pi_W\nu_n)=m$, then $W$ is a Galois subspace for $\nu_n$ if and only if $\Phi_{n,m}(W)$ belongs to one of the families presented in Theorem \ref{main} and Proposition \ref{exceptional}.
\end{proposition}

We will define $\Phi_{n,1}(\mathcal{C}_1)$ to be the set of all $W\in\mathbb{G}(n-2,n)$ such that $\pi_W\nu_n$ is an automorphism of $\mathbb{P}^1$; by the previous analysis it is easy to see that this family is $(n-1)$-dimensional.

\begin{theorem}
 Table \ref{classification} gives the list of all families of Galois subspaces for $\nu_n:\mathbb{P}^1\to\mathbb{P}^n$.
\end{theorem}

In particular, we see that the largest family is $\Phi_{n,2}^{-1}(\mathcal{C}_2)$ which is $n$-dimensional.

\begin{table}[]
\begin{tabular}{|c|c|c|c|}
 \hline& Families of Galois subspaces & Dimension of families & Number of families  \\\hline&&&\\
 $n=2,3$&$\Phi_{n,1}^{-1}(\mathcal{C}_1)$& $n-1$&$n$\\
 &$\Phi_{n,k}^{-1}(\mathcal{C}_k)$, $k=2,\ldots,n$  &$n-k+2$&\\&&&\\\hline&&&\\

 $4\leq n\leq 11$&$\Phi_{n,1}^{-1}(\mathcal{C}_1)$& $n-1$&$n+\lfloor\frac{n}{2}\rfloor-1$\\
 &$\Phi_{n,k}^{-1}(\mathcal{C}_k)$, $k=2,\ldots,n$& $n-k+2$&\\
 &$\Phi_{n,2k}^{-1}(\mathcal{D}_k)$, $k=3,\ldots,\lfloor\frac{n}{2}\rfloor$&$n-2k+3$&\\
 &$\Phi_{n,4}^{-1}(\mathcal{K})$&$n-1$&\\&&&\\\hline&&&\\
  $12\leq n\leq 23$&$\Phi_{n,1}^{-1}(\mathcal{C}_1)$& $n-1$&$n+\lfloor\frac{n}{2}\rfloor$\\
 &$\Phi_{n,k}^{-1}(\mathcal{C}_k)$, $k=2,\ldots,n$& $n-k+2$&\\
 &$\Phi_{n,2k}^{-1}(\mathcal{D}_k)$, $k=3,\ldots,\lfloor\frac{n}{2}\rfloor$&$n-2k+3$&\\
 &$\Phi_{n,4}^{-1}(\mathcal{K})$&$n-1$&\\
 &$\Phi_{n,12}^{-1}(\mathcal{A}_4)$&$n-9$&\\&&&\\\hline&&&\\
 $24\leq n\leq 59$&$\Phi_{n,1}^{-1}(\mathcal{C}_1)$& $n-1$&$n+\lfloor\frac{n}{2}\rfloor+1$\\
 &$\Phi_{n,k}^{-1}(\mathcal{C}_k)$, $k=2,\ldots,n$& $n-k+2$&\\
 &$\Phi_{n,2k}^{-1}(\mathcal{D}_k)$, $k=3,\ldots,\lfloor\frac{n}{2}\rfloor$&$n-2k+3$&\\
 &$\Phi_{n,4}^{-1}(\mathcal{K})$&$n-1$&\\
 &$\Phi_{n,12}^{-1}(\mathcal{A}_4)$&$n-9$&\\
 &$\Phi^{-1}(\mathcal{S}_4)$&$n-21$&\\&&&\\\hline&&&\\
  $n\geq60$&$\Phi_{n,1}^{-1}(\mathcal{C}_1)$& $n-1$&$n+\lfloor\frac{n}{2}\rfloor+2$\\
 &$\Phi_{n,k}^{-1}(\mathcal{C}_k)$, $k=2,\ldots,n$& $n-k+2$&\\
 &$\Phi_{n,2k}^{-1}(\mathcal{D}_k)$, $k=3,\ldots,\lfloor\frac{n}{2}\rfloor$&$n-2k+3$&\\
 &$\Phi_{n,4}^{-1}(\mathcal{K})$&$n-1$&\\
 &$\Phi_{n,12}^{-1}(\mathcal{A}_4)$&$n-9$&\\
 &$\Phi_{n,24}^{-1}(\mathcal{S}_4)$&$n-21$&\\
 &$\Phi_{n,60}^{-1}(\mathcal{A}_5)$&$n-57$&\\&&&\\\hline
\end{tabular}
\caption{List of all families of Galois subspaces}
\label{classification}
\end{table}


\clearpage







\begin{thebibliography}{999999}

\bibitem[Adr17]{Adrianov} N. Adrianov. \textit{Primitive monodromy grops of rational functions with one multiple pole.} J. Math. Sci. (N.Y.) 226 (2017), no. 5, 548-560.

\bibitem[Auf17]{Auff} R. Auffarth. \textit{A note on Galois embeddings of abelian varieties.} Manuscripta Math. 154, 279-284 (2017).

\bibitem[Bea10]{Beau} A. Beauville. \textit{Finite subgroups of $\text{PGL}_2(K)$}. Vector bundles and complex geometry, 23-29, Contemp. Math., 522, Amer. Math. Soc., Providence, RI, 2010.

\bibitem[Cuk99]{Cuk} F. Cukierman. \textit{Monodromy of projections.} Mat. Contemp. 16. 15th School of Algebra (Portuguese), 9-30.

\bibitem[GS95]{GS} R.M. Guralnick, J. Saxl. \textit{Monodromy groups of polynomials.} Groups of Lie type and their geometries. 125-160. Cambridge Univ. Press, Cambridge, 1995.

\bibitem[Kon17]{Konig} J. K\"{o}nig. \textit{On rational functions with monodromy group $M_{11}$.} J. Symbolic Comput. 79 (2017), part 2, 372-383.

\bibitem[Mul95]{Muller} P. M\"{u}ller. \textit{Primitive monodromy groups of polynomials.} Recent developments in the inverse Galois problem (Seattle, WA, 1993), 385-401, Contemp. Math., 186, Amer. Math. Soc., Providence, RI, 1995.

\bibitem[Mul98]{Muller2} P. M\"{u}ller. \textit{$(A_n,S_n)$ realizations by polynomials - on a question of Fried.} Finite Fields Appl. 4 (1998), 465-468.

\bibitem[PS05]{Pirola} G.P. Pirola, E. Schlesinger. \textit{Monodromy of projective curves.} J. Algebraic Geom. 14, 623-642 (2005).

\bibitem[Tak16]{Takahashi} T. Takahashi. \textit{Projection of a nonsingular plane quintic curve and the dihedral group of order eight.} Rend. Semin. Mat. Univ. Padova 135, 39-61 (2016)

\bibitem[Yos06]{Yoshi2} H. Yoshihara. \textit{Galois lines for space curves}. Algebra Collow. 13 (2006), no. 3, 455-469.

\bibitem[Yos07]{Yoshi} H. Yoshihara. \textit{Galois embedding of algebraic variety and its application to abelian surface.} Rend. Semin. Mat. Univ. Padova 117 (2007), 69-85.



\bibitem[Yos18]{Yoshiproblems} H. Yoshihara. \textit{List of problems.} url: http://hyoshihara.web.fc2.com/20170925.pdf. 


\end{thebibliography}
\end{document}